\documentclass[12pt]{amsart}
\usepackage{amsmath,amsfonts,amsthm,amssymb,mathrsfs,amsxtra,amscd,latexsym}
\usepackage[hmargin=2cm,vmargin=3cm]{geometry}
\usepackage[usenames]{color}
\usepackage[all]{xy}
\usepackage{setspace}

 \newtheorem{thm}{Theorem}[section]
 
 \newtheorem{cor}[thm]{Corollary}
 \newtheorem{lem}[thm]{Lemma}

 \newtheorem{rmk}[thm]{Remark}
 
 \theoremstyle{definition}
 \theoremstyle{remark}
 \numberwithin{equation}{section}

\newcommand{\sm}{\left(\begin{smallmatrix}}
\newcommand{\esm}{\end{smallmatrix}\right)}
\newcommand{\mat}{\left(\begin{matrix}}
\newcommand{\emat}{\end{matrix}\right)}

\def\CC{\mathbb{C}}

\def\QQ{\mathbb{Q}}

\def\ZZ{\mathbb{Z}}

\def\det{\mathrm{det}}

\def\SL{\mathrm{SL}}

\begin{document}

\title{Independence between coefficients of two modular forms}


 \author{Dohoon Choi}
 \author{Subong Lim}

\address{Department of Mathematics, Korea University, 145 Anam-ro, Seongbuk-gu, Seoul 02841, Republic of Korea}
\email{dohoonchoi@korea.ac.kr}

\address{Department of Mathematics Education, Sungkyunkwan University, Jongno-gu, Seoul 03063, Republic of Korea}
\email{subong@skku.edu}



\begin{abstract}
Let $k$ be an even integer and $S_k$ be the space of cusp forms of weight $k$ on $\SL_2(\ZZ)$.
Let $S = \oplus_{k\in 2\ZZ} S_k$.
For $f, g\in S$, we let $R(f, g)$ be the set of ratios of the Fourier coefficients of $f$ and $g$ defined by
$R(f,g) := \{ x \in \mathbb{P}^1(\CC)\ |\ x = [a_f(p) : a_g(p)]\ \text{for some prime $p$}\}$,
where $a_f(n)$ (resp. $a_g(n)$) denotes the $n$th Fourier coefficient of $f$ (resp. $g$).
In this paper, we prove that if $f$ and $g$ are nonzero and $R(f,g)$ is finite, then $f = cg$ for some constant $c$.
This result is extended to the space of weakly holomorphic modular forms on $\SL_2(\ZZ)$. We apply it to  study the number of representations of a positive integer by a quadratic form.
\end{abstract}

\maketitle

\section{Introduction} \label{section1}

The Fourier coefficients of a modular form play crucial roles in studying the theory of modular forms.
In particular, the $q$-expansion principle (for example, see \cite{DI} or \cite{Kat}) shows that a modular form is determined by its Fourier coefficients.
A natural question is to find relations between two modular forms when a connection between their Fourier coefficients is given.
It was proved  by Ramakrishnan
\cite[Appendix]{DK} that
if $f$ and $g$ are normalized Hecke eigenforms of the same weight such that for all primes $p$ outside a set $T$ of density $\delta(T)<\frac{1}{18}$
\begin{equation*} \label{DKcondition}
a_f(p)^2 = a_g(p)^2,
\end{equation*}
then there exists a quadratic character $\chi$ such that
\[
f = g \otimes \chi.
\]
Here, $a_f(n)$ (resp. $a_g(n)$) denotes the $n$th Fourier coefficient of $f$ (resp. $g$).
 This result was also known to Blasius and Serre.

Let $k$ be an even integer, and $M_{k}^{!}$ be the space of weakly holomorphic modular forms of weight $k$ on $\SL_2(\ZZ)$. Let
\[
M^{!} :=\oplus_{k \in 2\mathbb{Z}} M_{k}^{!}.
\]
Suppose that $f$ and $g$ are weakly holomorphic modular forms in $M^!$.
We define a subset $R(f,g)$ of $\mathbb{P}^1(\mathbb{C})$ by
\begin{equation} \label{ratioset}
R(f,g):=\left\{ x \in \mathbb{P}^1(\CC)\ |\ x = [a_f(p) : a_g(p)]\ \text{for some prime $p$}\right\}.
\end{equation}
This is a set of  ratios of the Fourier coefficients of $f$ and $g$.
For example, if $a_f(p)^2 = a_g(p)^2$ for every prime $p$, then  $R(f,g) = \{ [1:1], [1:-1]\}$.
Therefore, if $f$ and $g$ are Hecke eigenforms and  $R(f,g) = \{ [1:1], [1:-1]\}$,
then $f = cg$ for some constant $c$.
In this vein, the objective of this paper is to classify $f$ and $g$ in $M^{!}$ such that $R(f,g)$ is a finite set.
Our main  result is as follows.

\begin{thm} \label{main1}
Suppose that  $f$ and $g$ are nonzero weakly holomorphic modular forms in $M^!$.
If $R(f,g)$ is a finite set, then $f = cg$ for some constant $c$.
\end{thm}

This applies to  study the number of representations of an integer by a quadratic form.
Let $Q(x_1, \ldots, x_d)$  be a positive definite quadratic form over $\mathbb{Z}$ in $d$ variables with level one.
For a positive integer $n$, let
\[
r_{Q}(n):=\left|\{ (x_1, \ldots, x_d) \in \mathbb{Z}^d \; | \; Q(x_1, \ldots, x_d)=n \}\right|
\]
be the number of representations of the integer $n$ by a quadratic form $Q$.
Note  that 
\[
1+\sum_{n=1}^{\infty}r_{Q}(n)q^n
\]
is a modular form of weight $d/2$ on $\mathrm{SL}_2(\mathbb{Z})$.
Here, $q$ denotes $e^{2 \pi i z}$, where $z$ is a complex number whose imaginary part is positive.
Therefore, Theorem \ref{main1} gives the following corollary.

\begin{cor}
Assume that $Q_1(x_1, \ldots, x_d)$ and $Q_2(x_1, \ldots, x_d)$ are positive definite quadratic forms over $\mathbb{Z}$ in $d$ variables with level one.
If  there is a positive integer $n$ such that the numbers of representations of $n$ by
$Q_1$ and $Q_2$ are different, then the number of elements of the set
\[
R(Q_1,Q_2):=\left\{ x \in \mathbb{P}^1(\CC)\ |\ x = [r_{Q_1}(p) : r_{Q_2}(p)]\ \text{for some prime $p$}\right\}
\]
is infinite.
\end{cor}

The main ingredient of the proof of Theorem \ref{main1} is the result in \cite{Ri} on the Galois representations attached to Hecke eigenforms.
This is used to prove that if $f$ and $g$ are nonzero cusp forms on $\SL_2(\ZZ)$ and $R(f,g)$ is finite, then $f=cg$ for some constant $c$.
This is the main part of the proof of Theorem \ref{main1}.

\begin{rmk}
In the same way, the result can be extended to harmonic weak Maass forms.
In this case, if $f$ and $g$ are harmonic weak Maass forms whose shadows are cusp forms, then we need to look at the set
\[
R\left(f^+, g^+\right) := \left\{ x \in \mathbb{P}^1(\CC)\ |\ x = [a_{f^+}(p) : a_{g^+}(p)]\ \text{for some prime $p$}\right\},
\]
where $f^+$ (resp. $g^+$) denotes the holomorphic part of $f$ (resp. $g$) and $a_{f^+}(n)$ (resp. $a_{g^+}(n)$) denotes the $n$th Fourier coefficient of $f^+$ (resp. $g^+$).
\end{rmk}

The remainder of this paper is organized as follows.
In Section \ref{section2}, we review some preliminaries concerning the Fourier coefficients of weakly holomorphic modular forms and the Galois representations attached to Hecke eigenforms.
In Section \ref{section3}, we prove the main theorem for the case of cusp forms.
In Section \ref{section4}, we prove the main theorem: Theorem \ref{main1}.

\section{Preliminaries} \label{section2}
In this section, we review some basic material concerning the Fourier coefficients of weakly holomorphic modular forms and the Galois representations attached to Hecke eigenforms.

\subsection{Fourier coefficients of weakly holomorphic modular forms} \label{section2.1}
In this section, we review some results related to the asymptotic of the Fourier coefficients of weakly holomorphic modular forms based on \cite{Leh0}
and \cite{Ran}.

Let $f\in M^!_k$.
We write $k = 12o_k + k'$ with $o_k\in\ZZ$ and $k' \in \{0, 4, 6,8, 10, 14\}$.
Then, by the valence formula, we have
\begin{equation} \label{upperbound}
\mathrm{ord}_{\infty}(f) \leq o_k
\end{equation}
if $f$ is nonzero.
Moreover, for $m\geq -o_k$, there is a unique $f_{k,m} \in M^!_k$ such that
\begin{equation} \label{basis}
f_{k,m}(z) = q^{-m} + O(q^{o_k+1}).
\end{equation}
Then, $\{ f_{k,m}\ |\ m\geq-o_k\}$ forms a basis of $M^!_k$.
In \cite{DJ}, Duke and Jenkins studied various properties of this basis.

For positive integers $m,n$, and $c$, let
\[
A_{m,c}(n)  := \sum_{-d=0 \atop (c,d)=1}^{c-1} e^{\frac{2\pi i}{c}(nd-ma)},
\]
where $a$ is an integer such that $ad \equiv 1\ (\mathrm{mod}\ c)$.
We introduce the Bessel function of the first kind (for example, see \cite{Wat})
\begin{eqnarray*}
I_n(z) &:=& \sum_{t=0}^\infty \frac{(z/2)^{n+2t}}{t!\Gamma(n+t+1)}.
\end{eqnarray*}
Note that this Bessel function satisfies an asymptotic expansion
\[
I_n(z) \sim \frac{e^z}{\sqrt{2\pi z}}
\]
as $z\to\infty$.
Then, we have the following theorem.

\begin{thm} \label{weaklycoeff}
\label{poincare} \cite[Theorem 1--3]{Leh0}, \cite[pp. 149--151]{Ran}
For $m\geq -o_k$,
let $a_{k,m}(n)$ be the $n$th Fourier coefficient of $f_{k,m}$.
\begin{enumerate}
\item
If $m>0$, then
\[
a_{k,m}(n) \sim C_k A_{m,1}(n) \left(\frac{n}{m}\right)^{(k-1)/2}\frac{e^{4\pi \sqrt{mn}}}{(mn)^{1/4}}
\]
as $n\to\infty$, where $C_k$ is a constant dependent on $k$.

\item
If $m=0$, then
\[
a_{k,m}(n) \sim D_k n^{k-1}
\]
as $n\to\infty$, where $D_k$ is a constant dependent on $k$.
\end{enumerate}
\end{thm}

\subsection{Galois representations attached to Hecke eigenforms} \label{section2.2}
In this section, we introduce the result in \cite{Ri} concerning the Galois representations attached to Hecke eigenforms.
For a positive integer $n$, let $T_n$ be the $n$th Hecke operator.
For a positive even integer $k$, let $S_k$ denote the space of cusp forms of weight $k$ on $\SL_2(\ZZ)$.
Let $f(z) = \sum_{n=1}^\infty a_f(n)q^n \in S_k$ be a normalized Hecke eigenform, i.e. $f|T_p=a_f(p)f$ for every prime $p$ and $a_f(1) = 1$.
Let $E_f$ be the field generated by all the Hecke eigenvalues $a_f(p)$ over $\QQ$, and $H_f$ be the $\mathbb{Z}$-algebra generated by all the Hecke eigenvalues $a_f(p)$.
Let $G_{\QQ}$ denote $\mathrm{Gal}(\bar{\QQ}/\QQ)$.
For a prime $\ell$, let
\[
\rho_{f,\ell}:G_{\QQ} \rightarrow \mathrm{GL}_2\left(E_f \otimes \mathbb{Q}_{\ell}\right)
\]
be the representation of $G_{\QQ}$ attached to $f$.
Note that if $p$ is a prime not equal to $\ell$ and $\mathrm{Frob}_p \in G_{\QQ}$ is a Frobenius element at $p$, then the trace of $\rho_{f, \ell}(\mathrm{Frob}_p)$ is $a_f(p)$ in $E_f \otimes \QQ_{\ell}$ and the determinant of $\rho_{f, \ell}(\mathrm{Frob}_p)$ is $p^{k-1}$ (for example, see page 261 in \cite{Ri}).

Let $\mathcal{G}_{f,\ell}$ be the image of $\rho_{f,\ell}$ in $\mathrm{GL}_2\left(E_f \otimes \mathbb{Q}_{\ell}\right)$. If
\[
\mathcal{A}_{f,\ell}:= \left\{ u \in \mathrm{GL}_2\left(E_f \otimes \mathbb{Q}_{\ell}\right) \; | \; \det(u) \in (\mathbb{Z}_{\ell}^{\times})^{k-1}  \right\},
\]
then $\mathcal{A}_{f,\ell}$ contains $\mathcal{G}_{f,\ell}$. Moreover, it was proved in \cite{Ri} and \cite{Sw} that for all but finitely many primes $\ell$, we have
\[
\mathcal{G}_{f,\ell}=\mathcal{A}_{f,\ell}.
\]

Let $f' \in S_{k'}$ be a normalized Hecke eigenform.
Suppose that if $k=k'$, then $f$ and $f'$ are not conjugate under the action of $G_{\QQ}$. 
Let $T_{f,f'}$ be the $\mathbb{Z}$-subalgebra of $H_f \times H_{f'}$ generated by the pairs $\left(a_f(p),a_{f'}(p)\right)$.
It should be noted that according to the assumption, $\left[H_f \times H_{f'}:T_{f,f'}\right]$ is finite (for example, see lines 8-10 on p.268 in \cite{Ri}).
Let $\mathcal{A}_{\rho_f \times \rho_{f'}}$ be the image of $\rho_{f,\ell} \times \rho_{f',\ell}$ in $ \mathrm{GL}_2\left(E_f \otimes \mathbb{Q}_{\ell}\right) \times  \mathrm{GL}_2\left(E_{f'} \otimes \mathbb{Q}_{\ell}\right)$.
Ribet \cite{Ri} proved the following theorem.

\begin{thm}[Theorem 6.1 in \cite{Ri}]\label{Ribet1}
If $\ell$ is a prime such that
\begin{itemize}
\item $\ell \geq k+k'$,
\item $\mathcal{G}_{f,\ell}=\mathcal{A}_{f,\ell}$,
\item $\mathcal{G}_{f',\ell}=\mathcal{A}_{f',\ell}$,
\item $\ell \nmid \left[H_f \times H_{f'}:T_{f,f'}\right]$,
\end{itemize}
then
\[
\mathcal{A}_{\rho_f \times \rho_{f'}}= \left\{ (u,u') \in \mathcal{A}_{f,\ell} \times \mathcal{A}_{f',\ell} \; \big| \; \det(u)=v^{k-1}, \; \det(u')=v^{k'-1} \text{ for some } v \in \mathbb{Z}_{\ell}^{\times} \right\}.
\]
\end{thm}
This theorem implies the following lemma.

\begin{lem}\label{denseopen}
For integers $j$, $1 \leq j \leq m$, suppose that $f_j$ are normalized Hecke eigenforms in $S_{k_j}$ that are not conjugate to each other under the action of $G_{\QQ}$.
If $\ell$ is a sufficiently large prime, then $H_{f_1} \times \cdots \times H_{f_m} $ is dense in an open subset of $\left(H_{f_1} \otimes \QQ_{\ell}\right) \times \cdots \times \left(H_{f_m} \otimes \QQ_{\ell}\right)$.
\end{lem}

To prove this lemma, we need the following result.

\begin{lem}[Lemma 3.4 in \cite{Ri}]\label{multi}
Let $\mathcal{U}_1,\ldots,\mathcal{U}_t \; (t>1)$ be profinite groups.
Assume that for each $i$ the following condition is satisfied: for each open subgroup $\mathcal{W}$ of $\mathcal{U}_i$, the closure of the commutator subgroup of $\mathcal{W}$ is open in  $\mathcal{U}_i$. Let $\mathcal{H}$ be a closed subgroup of
\[
\mathcal{U}= \mathcal{U}_1 \times \cdots \times \mathcal{U}_t,
\]
which maps to an open subgroup of each group  $\mathcal{U}_i\times\mathcal{U}_j \; (i \neq j)$.
Then, $\mathcal{H}$ is open in $\mathcal{U}$.
\end{lem}

Now, we prove Lemma \ref{denseopen}.

\begin{proof}
Let $\mathcal{G}$ be the image of $\rho_{f_1, \ell} \times \cdots \times \rho_{f_m,\ell}$ in $\mathrm{GL}_2\left(H_{f_1} \otimes \ZZ_{\ell}\right) \times \cdots \times \mathrm{GL}_2\left(H_{f_m} \otimes \ZZ_{\ell}\right)$. We claim that $\mathcal{G}$ contains  $\mathrm{SL}_2\left(H_{f_1} \otimes \ZZ_{\ell}\right) \times \cdots \times \mathrm{SL}_2\left(H_{f_m} \otimes \ZZ_{\ell}\right)$, which then provides the proof.
Now, we prove the claim.
Note that if $\ell$ is a sufficiently large prime, then, for all pairs $\left(j_1,j_2\right)$, the prime $\ell$ satisfies the following conditions:
\begin{itemize}
\item $\mathcal{G}_{f_{j_1},\ell}=\mathcal{A}_{f_{j_1},\ell}$,
\item $\mathcal{G}_{f_{j_2},\ell}=\mathcal{A}_{f_{j_2},\ell}$,
\item $\ell \nmid \left[H_{f_{j_1}} \times H_{f_{j_2}}:T_{f_{j_1}, f_{j_2}} \right]$.
\end{itemize}
Therefore, we assume that $\ell$ satisfies these conditions.

Let
\[
\mathcal{U}:=\mathrm{SL}_2\left(H_{f_1} \otimes \ZZ_{\ell}\right) \times \cdots \times \mathrm{SL}_2\left(H_{f_m} \otimes \ZZ_{\ell}\right)
\]
and
\[
\mathcal{H}:=\mathcal{U} \cap \mathcal{G}.
\]
Note that $\mathcal{H}$ is closed in $\mathrm{GL}_2\left(H_{f_1} \otimes \QQ_{\ell}\right) \times \cdots \times \mathrm{GL}_2\left(H_{f_m} \otimes \QQ_{\ell}\right)$ since $G_{\QQ}$ is compact and  $\rho_{f_1,\ell} \times \cdots \times \rho_{f_m,\ell}$ is continuous.
Thus, we see that $\mathcal{H}$ is also closed in $\mathcal{U}$.
For all pairs $(j_1,j_2)$, the projection of $\mathcal{H}$ to $\mathrm{SL}_2\left(H_{f_{j_1}} \otimes \ZZ_{\ell}\right) \times \mathrm{SL}_2\left(H_{f_{j_2}} \otimes \ZZ_{\ell}\right)$ is surjective by Theorem \ref{Ribet1}.
For each $j$, we have
\[
\mathrm{SL}_2(H_{f_{j}} \otimes \ZZ_{\ell}) \cong \prod_{v | \ell} \mathrm{SL}_2(\mathcal{O}_v),
\]
where  $v$ denote the prime ideals of $H_{f_{j}}$ above $\ell$, and $\mathcal{O}_v$ denotes the completion of $H_{f_{j}}$ at $v$.
Note that $\mathrm{SL}_2(\mathcal{O}_v)$ is a $\ell$-adic lie group, and the lie algebra of $\mathrm{SL}_2(\mathcal{O}_v)$ is the same as its own derived algebra.
This implies that for each open subgroup $\mathcal{W}$ of $\mathrm{SL}_2(H_{f_{j}} \otimes \ZZ_{\ell})$, the closure of the commutator subgroup of $\mathcal{W}$ is open in  $\mathrm{SL}_2(H_{f_{j}} \otimes \ZZ_{\ell})$ (see Remark 3 on p. 253 in \cite{Ri}).
Therefore, by Lemma \ref{multi}, we complete the proof of the claim.
\end{proof}

\begin{rmk}
For the convenience of readers, let us recall the lie algebra of $\mathrm{SL}_2(\mathcal{O}_v)$ and its derived subalgebra.
The lie algebra of $\mathrm{SL}_2(\mathcal{O}_v)$ is isomorphic to
\[
\mathrm{sl}_2(\mathcal{O}_v):=\left \{ \left(
                                                                                     \begin{smallmatrix}
                                                                                       a & b \\
                                                                                       c & d \\
                                                                                     \end{smallmatrix}
                                                                                      \right)
\; |\ a,b,c,d \in \mathcal{O}_v \text{ and } a+c=0
\right \}.
\]
The derived subalgebra $\mathrm{Der}(\mathrm{sl}_2(\mathcal{O}_v))$ of $\mathrm{sl}_2(\mathcal{O}_v)$ is generated by all $[A,B] \; (A,B \in \mathrm{sl}_2(\mathcal{O}_v))$,
where $[A,B]=AB-BA$.
Note that
\[
   \left(
   \begin{smallmatrix}
    1 & 0 \\
    0 & -1 \\
   \end{smallmatrix}
\right)
=\left[\left(
   \begin{smallmatrix}
    0 & 1 \\
    0 & 0 \\
   \end{smallmatrix}
\right),
\left(
   \begin{smallmatrix}
    0 & 1 \\
    1 & 0 \\
   \end{smallmatrix}
\right) \right],
\]
\[
\left(
   \begin{smallmatrix}
    0 & 2 \\
    0 & 0 \\
   \end{smallmatrix}
\right)
=\left[\left(
   \begin{smallmatrix}
    0 & 1 \\
    0 & 0 \\
   \end{smallmatrix}
\right),
\left(
   \begin{smallmatrix}
   -1 & 0 \\
   0 & 1 \\
   \end{smallmatrix}
\right) \right],
\]
and
\[
   \left(
   \begin{smallmatrix}
    0 & 0 \\
    2 & 0 \\
   \end{smallmatrix}
\right)
=\left[\left(
   \begin{smallmatrix}
    0 & 0 \\
    1 & 0 \\
   \end{smallmatrix}
\right),
\left(
   \begin{smallmatrix}
   1 & 0 \\
   0 & -1 \\
   \end{smallmatrix}
\right) \right].
\]
Assume that $\ell \neq 2$. The matrices $\left(
   \begin{smallmatrix}
    1 & 0 \\
    0 & -1 \\
   \end{smallmatrix}
\right)$, $\left(
   \begin{smallmatrix}
    0 & 2 \\
    0 & 0 \\
   \end{smallmatrix}
\right)$, and $\left(
   \begin{smallmatrix}
    0 & 0 \\
    2 & 0 \\
   \end{smallmatrix}
\right)$ consist a basis of $\mathrm{sl}_2(\mathcal{O}_v)$.
Therefore, we have
$$
\mathrm{Der}(\mathrm{sl}_2(\mathcal{O}_v))=\mathrm{sl}_2(\mathcal{O}_v).
$$
\end{rmk}

\subsection{Lemma for hyperplanes} \label{section2.3}
For later use, we prove the following lemma.

\begin{lem} \label{notcontain}
Let $U$ be a subset of $V = \QQ^n$.
Suppose that $U$ is dense in an open subset of $V\otimes \QQ_{\ell}$.
\begin{enumerate}
\item
If $T_1, \ldots, T_m$ are hyperplanes in $V$, then
\[
U \not\subset \cup T_i.
\]

\item
If $L_1, \ldots, L_m$ are hyperplanes in $V\otimes \CC$, then we have
\[
U \not \subset \cup L_i.
\]
\end{enumerate}
\end{lem}

\begin{proof}
(1) Suppose that
\[
U \subset \cup T_i.
\]
This implies
\[
\overline{U} \subset \cup \overline{T_i}
\]
in $V\otimes \QQ_{\ell}$.
Then, $\cup \overline{T_i}$ contains an open set in $V\otimes \QQ_{\ell}$.
Note that $\overline{T_i}$ is a hyperplane in $V\otimes \QQ_{\ell}$ for each $i$.
This gives a contradiction.

(2) Due to (1), it is enough to prove that $L_i \cap V$ is contained in a hyperplane in $V$ for each $i$.
Note that $L_i$ can be expressed as
\[
\{(y_1, \ldots, y_n) \in V\otimes \CC|\     a_{i,1}y_1 + \cdots + a_{i,n}y_n = 0\}
\]
for $a_{i,j}\in\CC$.
Since $L_i$ is a hyperplane, we see that $(a_{i,1}, \ldots, a_{i,n}) \neq (0, \ldots, 0)$.
Without loss of generality, we may assume that $a_{i,n} \neq 0$.
Then, $(y_1, \ldots, y_n)\in L_i$ is equivalent to
\[
y_1 = t_1,\ \ldots,\ y_{n-1} = t_{n-1},\ y_n = -\frac{a_{i,1}}{a_{i,n}}t_1 - \cdots - \frac{a_{i,n-1}}{a_{i,n}}t_{n-1}
\]
for some $t_1, \ldots, t_{n-1}\in\CC$.
Therefore, $(x_1, \ldots, x_n) \in L_i \cap V$ is equivalent to
\begin{eqnarray*}
&&x_1 = t_1,\ \ldots,\ x_{n-1} = t_{n-1}\ \text{for some $t_1, \ldots, t_{n-1}\in\QQ$},\\
&&x_n = -\frac{a_{i,1}}{a_{i,n}}t_1 - \cdots - \frac{a_{i,n-1}}{a_{i,n}}t_{n-1} \in \QQ.
\end{eqnarray*}

We can  take $\QQ$-linearly independent complex numbers $\alpha_1 = 1, \ldots, \alpha_{n-1}$
such that
\[
\frac{a_{i,1}}{a_{i,n}}, \ldots, \frac{a_{i,n-1}}{a_{i,n}}
\]
can be expressed as $\QQ$-linear combinations of $\alpha_1, \ldots, \alpha_{n-1}$.
This implies that
\[
x_n = F_1(t_1, \ldots, t_{n-1}) + \sum_{i=2}^{n-1} \alpha_i F_i(t_1, \ldots, t_{n-1})
\]
for some degree $1$ homogeneous polynomials $F_i$ with coefficients in $\QQ$.
Therefore, $x_n\in \QQ$ is equivalent to $x_n = F_1(t_1, \ldots, t_{n-1})$.
From this, we see that $L_i \cap V$ is in the hyperplane
\[
\left\{ (x_1, \ldots, x_n)\in \QQ^n\ |\ x_n = F_1(x_1, \ldots, x_{n-1})\right\}.
\]
\end{proof}

\section{Coefficients of cusp forms} \label{section3}
In this section, we prove that if $f$ and $g$ are cusp forms and $R(f,g)$ is finite, then $f = cg$ for some constant $c$.
To prove this, we need the following lemmas.

\begin{lem}  \label{nonconjugate}
Suppose that $f_1, \ldots, f_m$ are normalized Hecke eigenforms on $\SL_2(\ZZ)$ such that any two of them are not conjugate under the action of $G_{\QQ}$.
Then, there are no finite subsets $B$ of $(E_{f_1}\otimes \CC) \times \cdots \times (E_{f_m} \otimes \CC)$ such that $(0, \ldots, 0) \not\in B$ and for any primes $p$, we have
\[
A_1 a_{f_1}(p) + \cdots + A_m a_{f_m}(p) = 0
\]
for some $(A_1, \ldots, A_m) \in B$.
\end{lem}

\begin{proof}
Suppose that $B$ is a finite subset of $(E_{f_1}\otimes \CC) \times \cdots \times (E_{f_m} \otimes \CC)$ such that $(0, \ldots, 0) \not\in B$ and for any primes $p$, we have
\[
A_1 a_{f_1}(p) + \cdots + A_m a_{f_m}(p) = 0
\]
for some $(A_1, \ldots, A_m) \in B$.
Then, the set
\[
C = \bigcup_{(A_1, \ldots, A_m)\in B} \{ (x_1, \ldots, x_m) \in (E_{f_1}\otimes \CC) \times \cdots \times (E_{f_m}\otimes \CC)\ |\ A_1x_1 + \cdots + A_mx_m = 0\}
\]
is a finite union of hyperplanes $(E_{f_1}\otimes \CC) \times \cdots \times (E_{f_m}\otimes \CC)$.

By Lemma \ref{denseopen}, we see that there is a prime $\ell$ such that
$H_{f_1}\times \cdots \times H_{f_{m}}$ is dense in an open subset of $(H_{f_1}\otimes \QQ_{\ell}) \times \cdots \times (H_{f_m}\otimes \QQ_{\ell})$.
Since the trace of  $\rho_{f_i,\ell}(\mathrm{Frob}_p)$ is $a_{f_i}(p)$ for primes $p\neq \ell$ and $\{ \mathrm{Frob}_p\ |\ \text{$p$ is a prime}\}$ is dense in $G_{\QQ}$ by Chebotarev's density theorem, we see that the set
\[
T = \{ (a_{f_1}(p), \ldots, a_{f_n}(p))\ |\ \text{$p$ is a prime with $p\neq \ell$} \}
\]
is a dense subset of $H_{f_1}\times \cdots \times H_{f_{m}}$.
Since we have
\[
(H_{f_1}\otimes \QQ_{\ell}) \times \cdots \times (H_{f_m}\otimes \QQ_{\ell}) \cong (E_{f_1} \times \cdots \times E_{f_m}) \otimes \QQ_{\ell}
\]
and
\[
(E_{f_1}\otimes \CC) \times \cdots \times (E_{f_n}\otimes \CC) \cong (E_{f_1} \times \cdots \times E_{f_m}) \otimes \CC,
\]
by Lemma \ref{notcontain}, the set $T$ is not contained in any finite union of hyperplanes in $(E_{f_1}\otimes \CC) \times \cdots \times (E_{f_m} \otimes \CC)$.
This is a contradiction since $T$ is contained in $C$ by the assumption.
\end{proof}

Let $f$ be a normalized Hecke eigenform on $\SL_2(\ZZ)$.
Let $m = [E_f:\QQ]$ and $\sigma_1, \ldots, \sigma_m$ be the embeddings from $E_f$ to $\bar{\QQ}$.
Note that $E_f \otimes \QQ_{\ell} \cong (\QQ_{\ell})^m$.
Let $\{\tau_1, \ldots, \tau_m\}$ be a basis of $E_f$ over $\QQ$.
Then, for each $n>0$, the coefficient $a_f(n)$ can be written as a linear combination of $\tau_1, \ldots, \tau_m$, i.e.,
\[
a_f(n) = a_1(n)\tau_1 + \cdots + a_m(n)\tau_m
\]
for $a_i(n)\in\QQ$.
This means that
\begin{equation} \label{decomposition}
a_f(n)^{\sigma} = a_1(n)\tau_1^{\sigma} + \cdots + a_m(n)\tau_m^{\sigma}
\end{equation}
for $\sigma\in G_{\QQ}$.
From this, we prove the following lemma.

\begin{lem} \label{conjugate}
There are no finite subsets $B$ of $\CC^m$ such that $(0, \ldots, 0)\not\in B$ and for any primes $p$, we have
\begin{equation} \label{sigmaequation}
A_1 a_f(p)^{\sigma_1} + \cdots + A_m a_f(p)^{\sigma_m} = 0
\end{equation}
for some $(A_1, \ldots, A_m)\in B$.
\end{lem}

\begin{proof}
Suppose that $B$ is a finite subset of $\CC^m$ such that $(0, \ldots, 0)\not\in B$ and for any primes $p$, we have
\[
A_1 a_f(p)^{\sigma_1} + \cdots + A_m a_f(p)^{\sigma_m} = 0
\]
for some $(A_1, \ldots, A_m)\in B$.
By (\ref{decomposition}), we see that the equation (\ref{sigmaequation}) is equivalent to
\[
(A_1,  \ldots, A_m)
\sm \tau_1^{\sigma_1} & \cdots & \tau_m^{\sigma_1} \\
        \vdots                    & \ddots & \vdots                    \\
        \tau_1^{\sigma_m} & \cdots & \tau_m^{\sigma_m} \esm
\sm a_1(p) \\ \vdots \\ a_m(p) \esm
= 0.
\]
Since $\{\tau_1, \ldots, \tau_m\}$ is a basis of $E_f$ over $\QQ$, the matrix
\[
\sm \tau_1^{\sigma_1} & \cdots & \tau_m^{\sigma_1} \\
        \vdots                    & \ddots & \vdots                    \\
        \tau_1^{\sigma_m} & \cdots & \tau_m^{\sigma_m} \esm
\]
is invertible.
Therefore, we have
\[
(A'_1, \ldots, A'_m) =  (A_1, \ldots,  A_m)
\sm \tau_1^{\sigma_1} & \cdots & \tau_m^{\sigma_1} \\
        \vdots                    & \ddots & \vdots                    \\
        \tau_1^{\sigma_m} & \cdots & \tau_m^{\sigma_m} \esm
\neq (0, \ldots, 0)
\]
since $(A_1, \ldots, A_m) \neq (0, \ldots, 0)$.
This means that the set
\[
\{(a_1(p), \ldots, a_m(p))\ |\ \text{$p$ is a prime}\}
\]
is a subset of the set
\begin{equation} \label{closedset}
C = \bigcup_{(A_1, \ldots, A_m)\in B} \{(x_1, \ldots, x_m) \in \CC^m\ |\ A'_1x_1 + \cdots + A'_mx_m = 0\},
\end{equation}
which is a finite union of hyperplanes in $E_f \otimes \CC$.

By Lemma \ref{denseopen}, we see that there is a prime $\ell$ such that
 $H_{f}$ is dense in an open subset of $H_{f}\otimes \QQ_{\ell}$.
Since  the trace of $\rho_{f,\ell}(\mathrm{Frob}_p)$ is $a_{f}(p)$ for primes $p\neq \ell$ and $\{ \mathrm{Frob}_p\ |\ \text{$p$ is a prime}\}$ is dense in $G_{\QQ}$ by Chebotarev's density theorem, we see that the set
\[
T = \{ a_f(p)\ |\ \text{$p$ is a prime with $p\neq \ell$}\}
\]
is a dense subset of $H_f$.
Then, by Lemma \ref{notcontain}, the set $T$ is not contained in any finite union of hyperplanes in $E_{f}\otimes \CC$.
We consider the isomorphism from $E_f$ to $\QQ^m$ defined by $a\mapsto (a_1, \ldots, a_m)$, where $a_1, \ldots, a_m$ is determined by the decomposition
\[
a = a_1 \tau_1 + \cdots + a_m\tau_m.
\]
By using this isomorphism, we see that the set
\[
T' = \{ (a_1(p), \ldots, a_m(p))\ |\ \text{$p$ is a prime with $p\neq\ell$}\}
\]
is not contained in any finite union of hyperplanes in $\CC^m$.
This is a contradiction since  $T'$ is contained in $C$ by the assumption.
\end{proof}

Suppose that $f_1, \ldots, f_m$ are normalized Hecke eigenforms such that any two of them are not conjugate under the action of $G_{\QQ}$.
Let $a_i(n)$ be the $n$th Fourier coefficient of $f_i$, i.e.,
\[
f_i(z) = \sum_{n>0} a_i(n)q^n.
\]
For each $i$, let $t_i = [E_{f_i} : \QQ]$ and $\{ \sigma_{i, 1}, \ldots, \sigma_{i,t_i}\}$ be the embeddings of $E_{f_i}$ to $\bar{\QQ}$.
Let $\{\tau_{i,1}, \ldots, \tau_{i,t_i}\}$ be a basis of $E_{f_i}$ over $\QQ$.
Therefore, $a_i(n)$ can be written as a linear combination of $\tau_{i,1}, \ldots, \tau_{i,t_i}$, i.e.,
\[
a_i(n) = a_{i,1}(n)\tau_{i,1} + \cdots + a_{i,t_i}(n)\tau_{i,t_i}
\]
for $a_{i,j}(n)\in\QQ$.
From this, we see that
\begin{equation} \label{decomposition2}
a_i(n)^{\sigma} = a_{i,1}(n)\tau_{i,1}^{\sigma} + \cdots + a_{i,t_i}(n)\tau_{i,t_i}^{\sigma}
\end{equation}
for $\sigma\in G_{\QQ}$.
Then, we have the following lemma.

\begin{lem} \label{mixed}
Let $t = t_1 + \cdots + t_m$.
There are no finite subsets $B$ of $\CC^t$ such that $(0, \ldots, 0)\not \in B$ and for any  primes $p$, we have
\begin{equation} \label{sigmaequation2}
\sum_{i=1}^m \sum_{j=1}^{t_i} A_{i,j} a_i(p)^{\sigma_{i,j}} = 0
\end{equation}
for some $(A_{1, 1}, \ldots, A_{1, t_1}, \ldots, A_{m,1}, \ldots, A_{m,t_m})\in B$.
\end{lem}

\begin{proof}
Suppose that $B$ is a finite subset of $\CC^t$ such that $(0, \ldots, 0)\not \in B$ and for any primes $p$, we have
\[
\sum_{i=1}^m \sum_{j=1}^{t_i} A_{i,j} a_i(p)^{\sigma_{i,j}} = 0
\]
for some $(A_{1, 1}, \ldots, A_{1, t_1}, \ldots, A_{m,1}, \ldots, A_{m,t_m})\in B$.

By (\ref{decomposition2}), we see that the equation (\ref{sigmaequation2}) is equivalent to
\[
(A_{1,1}, \ldots, A_{1,t_1}, \ldots, A_{m,1}, \ldots, A_{m,t_m})
\sm \tau_{1,1}^{\sigma_{1,1}} & \cdots & \tau_{1, t_1}^{\sigma_{1,1}}   &            &            &             &                                          &             &                                           \\
        \vdots                               & \ddots & \vdots                                    &            & 0          &             &                                         &   0         &                                            \\
        \tau_{1,1}^{\sigma_{1,t_1}} & \cdots & \tau_{1, t_1}^{\sigma_{1, t_1}} &            &            &             &                                          &             &                                            \\
                                                  &             &                                              & \ddots &            &             &                                          &             &                                           \\
                                                  &  0         &                                              &             & \ddots &            &                                          & 0           &                                           \\
                                                  &             &                                              &              &           & \ddots &                                          &             &                                            \\
                                                  &             &                                             &               &           &            & \tau_{m,1}^{\sigma_{m,1}} & \cdots & \tau_{m,t_m}^{\sigma_{m,1}} \\
                                                  & 0          &                                              &               & 0        &            & \vdots                               & \ddots  & \vdots                                 \\
                                                  &            &                                              &               &           &            & \tau_{m,1}^{\sigma_{m,t_m}} & \cdots & \tau_{m,t_m}^{\sigma_{m,t_m}} \esm
\sm a_{1,1}(p) \\ \vdots \\ a_{1,t_1}(p) \\ \vdots \\ a_{m,1}(p) \\ \vdots \\ a_{m,t_m}(p)\esm
= 0.
\]
Note that the matrix
\[
\sm \tau_{1,1}^{\sigma_{1,1}} & \cdots & \tau_{1, t_1}^{\sigma_{1,1}}   &            &            &             &                                          &             &                                           \\
        \vdots                               & \ddots & \vdots                                    &            & 0          &             &                                         &   0         &                                            \\
        \tau_{1,1}^{\sigma_{1,t_1}} & \cdots & \tau_{1, t_1}^{\sigma_{1, t_1}} &            &            &             &                                          &             &                                            \\
                                                  &             &                                              & \ddots &            &             &                                          &             &                                           \\
                                                  &  0         &                                              &             & \ddots &            &                                          & 0           &                                           \\
                                                  &             &                                              &              &           & \ddots &                                          &             &                                            \\
                                                  &             &                                             &               &           &            & \tau_{m,1}^{\sigma_{m,1}} & \cdots & \tau_{m,t_m}^{\sigma_{m,1}} \\
                                                  & 0          &                                              &               & 0        &            & \vdots                               & \ddots  & \vdots                                 \\
                                                  &            &                                              &               &           &            & \tau_{m,1}^{\sigma_{m,t_m}} & \cdots & \tau_{m,t_m}^{\sigma_{m,t_m}} \esm
\]
is invertible since the matrix
\[
\sm \tau_{i,1}^{\sigma_{i,1}} & \cdots & \tau_{i, t_i}^{\sigma_{i,1}} \\
        \vdots                               & \ddots & \vdots                               \\
        \tau_{i,1}^{\sigma_{i,t_i}} & \cdots & \tau_{i, t_i}^{\sigma_{i, t_i}} \esm
\]
is invertible for each $i$.
Therefore, if we let
\begin{eqnarray*}
&& (A'_{1,1}, \ldots, A'_{1,t_1}, \ldots, A'_{m,1}, \ldots, A'_{m,t_m})\\
&& = (A_{1,1}, \ldots, A_{1,t_1}, \ldots, A_{m,1}, \ldots, A_{m,t_m})
\sm \tau_{1,1}^{\sigma_{1,1}} & \cdots & \tau_{1, t_1}^{\sigma_{1,1}}   &            &            &             &                                          &             &                                           \\
        \vdots                               & \ddots & \vdots                                    &            & 0          &             &                                         &   0         &                                            \\
        \tau_{1,1}^{\sigma_{1,t_1}} & \cdots & \tau_{1, t_1}^{\sigma_{1, t_1}} &            &            &             &                                          &             &                                            \\
                                                  &             &                                              & \ddots &            &             &                                          &             &                                           \\
                                                  &  0         &                                              &             & \ddots &            &                                          & 0           &                                           \\
                                                  &             &                                              &              &           & \ddots &                                          &             &                                            \\
                                                  &             &                                             &               &           &            & \tau_{m,1}^{\sigma_{m,1}} & \cdots & \tau_{m,t_m}^{\sigma_{m,1}} \\
                                                  & 0          &                                              &               & 0        &            & \vdots                               & \ddots  & \vdots                                 \\
                                                  &            &                                              &               &           &            & \tau_{m,1}^{\sigma_{m,t_m}} & \cdots & \tau_{m,t_m}^{\sigma_{m,t_m}} \esm,
\end{eqnarray*}
then $(A'_{1,1}, \ldots, A'_{1,t_1}, \ldots, A'_{m,1}, \ldots, A'_{m,t_m}) \neq (0, \ldots, 0)$ since $(A_{1,1}, \ldots, A_{1,t_1}, \ldots, A_{m,1}, \ldots, A_{m,t_m}) \neq (0, \ldots, 0)$.
This means that the set
\[
\{(a_{1,1}(p), \ldots, a_{1,t_1}(p), \ldots, a_{m,1}(p), \ldots, a_{m,t_m}(p))\ |\ \text{$p$ is a prime}\}
\]
is a subset of the set
\begin{equation} \label{closedset2}
C = \bigcup_{(A_{1,1}, \ldots, A_{1,t_1}, \ldots, A_{m,1}, \ldots, A_{m,t_m})\in B} \left\{(x_{1,1}, \ldots, x_{1,t_1}, \ldots, x_{m,1},\ldots, x_{m,t_m})\in \CC^t\ \bigg|\ \sum_{i=1}^m \sum_{j=1}^{t_i} A'_{i,j}x_{i,j} = 0\right\},
\end{equation}
which is a finite union of hyperplanes in $\CC^t$.

As in the proof of Lemma \ref{nonconjugate}, there is a prime $\ell$ such that the set
\[
T = \{ (a_{f_1}(p), \ldots, a_{f_n}(p))\ |\ \text{$p$ is a prime with $p\neq \ell$} \}
\]
is not contained in any finite union of hyperplanes in $(E_{f_1}\otimes \CC) \times \cdots \times (E_{f_m} \otimes \CC)$.
As in the proof of Lemma \ref{conjugate}, we consider the isomorphism from $E_{f_1} \times \cdots \times E_{f_m}$ to $\QQ^t$ defined by
\[
(a_1, \ldots, a_m) \mapsto (a_{1,1}, \ldots, a_{1, t_1}, \ldots, a_{m,1},\ldots, a_{m,t_m}),
\]
where $a_{1,1}, \ldots, a_{1, t_1}, \ldots, a_{m,1},\ldots, a_{m,t_m}$ are determined by the decomposition
\[
a_i = a_{i,1} \tau_{i,1} + \cdots + a_{i,t_i}\tau_{i,t_i}
\]
for each $i$.
By this isomorphism, we see that
\[
T' = \{(a_{1,1}(p), \ldots, a_{1,t_1}(p), \ldots, a_{m,1}(p), \ldots, a_{m,t_m}(p))\ |\ \text{$p$ is a prime with $p\neq \ell$}\}
\]
is not contained in any finite union of hyperplanes in $\CC^t$.
This is a contradiction since  $T'$ is contained in $C$ by the assumption.
\end{proof}

From Lemma \ref{mixed}, we can prove the following theorem.

\begin{thm} \label{cuspcase}
Suppose that $f$ and $g$ are nonzero cusp forms on $\SL_2(\ZZ)$.
If $f$ is not a constant multiple of $g$, then
$R(f,g)$ defined in (\ref{ratioset}) is not finite.
\end{thm}

\begin{proof}
Since $f$ and $g$ are cusp forms, each function can be written as a linear combination of finitely many normalized Hecke eigenforms $\{f_1, \ldots, f_m\}$, i.e.,
\[
f = \sum_{i=1}^m a_if_i,\ g = \sum_{i=1}^m b_i f_i
\]
for $a_i, b_i \in \CC$.
For each $i$, we consider embeddings $\sigma_{i,1}, \ldots, \sigma_{i,t_i}$ from $E_{f_i}$ to $\bar{\QQ}$, where $t_i = [E_{f_i} : \QQ]$.
Consider the set
\[
A = \{ f_{i}^{\sigma_{i,j}}\ |\ 1\leq i\leq m,\ 1\leq j\leq t_i\}.
\]
We write
\[
t = |A|,\ A = \{h_1, \ldots, h_t\}.
\]
Since $A$ contains $\{f_1, \ldots, f_m\}$, both $f$ and $g$ can be written as linear combinations of elements of $A$, i.e.,
\[
f = \sum_{i=1}^t \alpha_i h_i,\ g = \sum_{i=1}^t \beta_i h_i
\]
for $\alpha_i, \beta_i\in \CC$.

Suppose that $R(f,g)$ is finite and that $f$ is not a constant multiple of $g$.
Let $a_i(n)$ be the $n$th Fourier coefficient of $h_i$.
Then, there is a finite subset $B$ of $\CC^t$ such that $(0, \ldots, 0)\not\in B$ and for each prime $p$, we have
\[
A_1 a_1(p) + \cdots + A_t a_t(p) = 0
\]
for some $(A_1, \ldots, A_t)\in B$.
This is a contradiction due to Lemma \ref{mixed}.
\end{proof}

\begin{rmk} \label{cusprmk}
Suppose that $f$ and $g$ are cusp forms.
They may be zero.
Then, Theorem \ref{cuspcase} implies that if $R(f,g)$ is finite, then there are constants $\alpha$ and $\beta$ such that  $\alpha f = \beta g$.
\end{rmk}

\section{Proof of the main result} \label{section4}
In this section, we prove  Theorem \ref{main1}.
Suppose that $f$ is not a cusp form and $g$ is a cusp form.
For a weakly holomorphic modular form $h\in M^!$,
let $m_h = |\mathrm{ord}_{\infty}(h)|$ and $k_h$ be the weight of $h$.
Since $g$ is a cusp form, the Fourier coefficients of $g$ should satisfy the Hecke bound
\[
a_g(n) = O(n^{k_g/2})
\]
as $n\to\infty$.
Then, by Theorem \ref{weaklycoeff}, we have
\[
\lim_{n\to\infty} \left| \frac{a_f(n)}{a_g(n)}\right| = \infty,
\]
which means that the set $R(f,g)$ cannot be finite.
This is a contradiction.
Therefore, both $f$ and $g$ are cusp forms or  none of them are cusp forms.
If both $f$ and $g$ are cusp forms, then by Theorem \ref{cuspcase}, $f = cg$ for some constant $c$.

Suppose that neither $f$ nor $g$ is a cusp form.
If  $m_f \neq m_g$, then we may assume that $m_f > m_g$.
By Theorem \ref{weaklycoeff}, we have
\[
\lim_{n\to\infty} \left|\frac{a_f(n)}{a_g(n)}\right| = \infty,
\]
which means that $R(f,g)$ is an infinite set.
This is a contradiction.
Therefore, $m_f = m_g$.
In the same way, we see that $k_f = k_g$.
By multiplying a nonzero constant to $g$,
we may assume that
\begin{equation} \label{constantassumption}
a_f(-m_f) = a_g(-m_g).
\end{equation}
Therefore, it is enough to show that $f = g$.

Suppose that $[\alpha:\beta]\in R(f,g)$ satisfying
\begin{equation} \label{subsequence}
\beta a_f(p) = \alpha a_g(p)
\end{equation}
for infinitely many primes $p$.
Such $[\alpha:\beta]$ exists since $R(f,g)$ is finite.
By the asymptotic expansion given in Theorem \ref{weaklycoeff}, we see that both $f$ and $g$ have only finitely many zero coefficients.
Therefore, both $\alpha$ and $\beta$ are nonzero.

Then, we have a strictly increasing sequence $\{ p_i\}$ of primes satisfying (\ref{subsequence}).
By Theorem \ref{weaklycoeff} and (\ref{subsequence}), we have
\begin{equation} \label{infiniteratio}
\frac{\alpha}{\beta} = \lim_{i\to\infty} \frac{a_f(p_i)}{a_g(p_i)} = \lim_{i\to\infty} \frac{a_f(-m_f) e^{4\pi\sqrt{p_i m_f}}} {a_g(-m_g) e^{4\pi\sqrt{p_i m_g}}} = 1
\end{equation}
since $m_f = m_g$ and $k_f = k_g$.
Therefore, we see that $[\alpha:\beta] = [1:1] \in R(f,g)$.

This implies that if $[\alpha:\beta] \in R(f,g)$ and $[\alpha:\beta]\neq [1:1]$, then the number of primes satisfying (\ref{subsequence}) is finite.
Since $R(f,g)$ is finite, we see that
\[
a_f(p) = a_g(p)
\]
for all but finitely many primes $p$.

Now, we prove that $f-g$ is a cusp form.
Suppose that $f - g$ is not a cusp form.
This means that the principal parts of $f$ and $g$ are not the same.
Let
\[
m_0 = \mathrm{max}\{ n\geq 0\ |\ a_f(-n) \neq a_g(-n)\}.
\]
Then, by (\ref{upperbound}), we see that $-m_0 \leq o_k$.
Let
\begin{eqnarray*}
\hat{f} &=& \sum_{m>m_0} a_f(-m)f_{k_f,m},\ \tilde{f} = f - \hat{f},\\
\hat{g} &=& \sum_{m>m_0} a_g(-m)f_{k_f,m},\ \tilde{g} = g - \hat{g},
\end{eqnarray*}
where $f_{k_f,m}$ is a weakly holomorphic modular form as in (\ref{basis}).
We denote by $a_{\hat{f}}(n)$ (resp. $a_{\tilde{f}}(n), a_{\hat{g}}(n)$, and $a_{\tilde{g}}(n)$) the $n$th Fourier coefficient of $\hat{f}$ (resp. $\tilde{f}, \hat{g}$, and $\tilde{g}$).
By the definition of $m_0$, we see that $\hat{f} = \hat{g}$.
Then, for all but finitely many primes $p$, we have
\begin{equation} \label{tildesub}
a_{\tilde{f}}(p) = a_{\tilde{g}}(p).
\end{equation}
This implies that $R(\tilde{f}, \tilde{g})$ is finite.

If $m_0>0$, then at least one of $\tilde{f}$ and $\tilde{g}$ is not a cusp form.
Since $R(\tilde{f}, \tilde{g})$ is finite, in the same way as above, we see that $m_{\tilde{f}} = m_{\tilde{g}} = m_0$.
Note that we have a strictly increasing sequence $\{p_i\}$ of primes satisfying (\ref{tildesub}).
Then, by the same argument as in (\ref{infiniteratio}), we have
\[
1 = \lim_{i\to\infty} \frac{a_{\tilde{f}}(p_i)}{a_{\tilde{g}}(p_i)} = \lim_{i\to\infty} \frac{a_{\tilde{f}}(-m_0) e^{4\pi\sqrt{p_i m_0}}}{a_{\tilde{g}}(-m_0) e^{4\pi\sqrt{p_i m_0}}} = \frac{a_{\tilde{f}}(-m_0)}{a_{\tilde{g}}(-m_0)}  .
\]
Therefore, we obtain $a_{\tilde{f}}(-m_0) = a_{\tilde{g}}(-m_0)$ which implies that $a_f(-m_0) = a_g(-m_0)$.
This is a contradiction due to the definition of $m_0$.

If $m_0 = 0$, then  both $\tilde{f}$ and $\tilde{g}$ are holomorphic modular forms
and $a_{\tilde{f}}(0) \neq a_{\tilde{g}}(0)$.
Then, we see that
\begin{eqnarray*}
\tilde{f} &=& a_{\tilde{f}}(0)f_{k_f,0} + f_c,\\
\tilde{g} &=& a_{\tilde{g}}(0)f_{k_f,0} + g_c
\end{eqnarray*}
for some cusp forms $f_c, g_c\in S_{k_f}$.
By the Hecke bound,  cusp forms $f_c(z) = \sum_{n>0} a_{f_c}(n)q^n$ and $g_c(z) = \sum_{n>0} a_{g_c}(n)q^n$  satisfy
\begin{equation} \label{cuspasym}
a_{f_c}(n), a_{g_c}(n) = O(n^{k_f/2}).
\end{equation}
By (\ref{tildesub}), we have a strictly increasing sequence $\{p_i\}$ of primes satisfying
\[
(a_{\tilde{f}}(0) - a_{\tilde{g}}(0)) a_{k_f,0}(p_i) = a_{g_c}(p_i) - a_{f_c}(p_i),
\]
where $a_{k_f,0}(n)$ is the $n$th Fourier coefficient of $f_{k_f,0}$.
This is a contradiction since
\[
\lim_{i\to\infty} \left| \frac{(a_{\tilde{f}}(0) - a_{\tilde{g}}(0)) a_{k_f,0}(p_i)} {a_{g_c}(p_i) - a_{f_c}(p_i)}\right| = \infty
\]
by Theorem \ref{weaklycoeff} and (\ref{cuspasym}).
In conclusion, $f-g$ should be a cusp form.

If $S_{k_f} = \{0\}$, the proof is completed.
Otherwise, we consider
\begin{eqnarray*}
F &=& f - \sum_{m\geq 0 } a_f(-m) f_{k_f, m},\\
G &=& g - \sum_{m\geq 0} a_g(-m) f_{k_f, m}.
\end{eqnarray*}
Both $F$ and $G$ are cusp forms and
\begin{equation} \label{cuspratio}
a_F(p) = a_G(p)
\end{equation}
for all but finitely many primes $p$, where $a_F(n)$ (resp. $a_G(n)$) denotes the $n$th Fourier coefficient of $F$ (resp. $G$).
Therefore, $R(F,G)$ is finite.
By Theorem \ref{cuspcase}, there are constants $\alpha',\beta'$ such that  $\beta' F = \alpha' G$.
By (\ref{cuspratio}), we see that
$F = G$, and this implies that $f = g$ since $a_f(-m) = a_g(-m)$ for all $m\geq 0$.
This completes the proof.

\section*{Acknowledgments}
The authors are grateful to the referee for  helpful comments and  corrections.
The authors also thank Jeremy Rouse for  useful comments on the previous version of this paper.



 
\end{document}